\documentclass[10pt]{article}

\usepackage{amsthm}
\usepackage{amsmath,amssymb}
\usepackage[dvips]{graphicx}
\usepackage{float}
\usepackage{wrapfig}
\usepackage{subcaption}

\newtheorem{theorem}{Theorem}[section]

\theoremstyle{definition}
\newtheorem{definition}[theorem]{Definition}
\newtheorem{lemma}[theorem]{Lemma}
\newtheorem{remark}[theorem]{Remark}

\begin{document}

\title{On ordered sequences for link diagrams with respect to Reidemeister moves I and III}
\author{Kishin Sasaki}
\maketitle

\begin{abstract}
We first prove that, infinitely many pairs of trivial knot diagrams that are transformed into each other by applying Reidemeister moves I and III are NOT transformed into each other by a sequence of the Reidemeister moves I that increase the number of crossings, followed by a sequence of Reidemeister moves III, followed by a sequence of the Reidemeister moves I that decrease the number of crossings. To create a simple sequence between link diagrams that are transformed into each other by applying finitely many Reidemeister moves I and III, we prove that the link diagrams are always transformed into each other by applying an I-generalized ordered sequence. 
\end{abstract}

\section{Introduction}

In \cite{reidemeister}, K. Reidemeister introduced Reidemeister moves, which are local moves on link diagrams, and proved that link diagrams of equivalent links are transformed into each other by applying finitely many Reidemeister moves (and vice versa). Figure \ref{fig:reidemeister+-} shows all the Reidemeister moves, where $\uparrow$ and $\downarrow$ are assigned to the moves that increase and decrease the numbers of crossings, respectively.

\begin{figure}[H]
 \begin{center}
 \includegraphics[width=12cm,keepaspectratio]{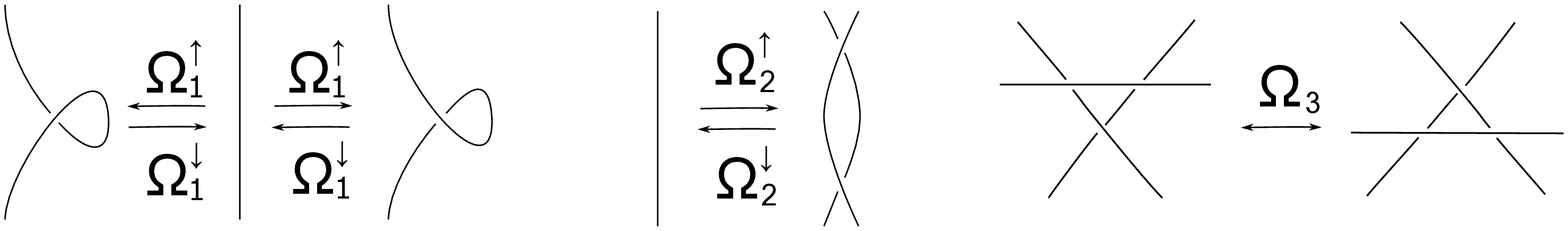}
 \end{center}
 \caption{Reidemeister moves}
 \label{fig:reidemeister+-}
\end{figure}

In \cite{ordered_sequence}, A. Coward proved that link diagrams of equivalent links are always transformed into each other by applying a sequence of $\Omega_{1}^{\uparrow}$ moves, followed by a sequence of $\Omega_{2}^{\uparrow}$ moves, followed by a sequence of $\Omega_{3}$ moves, followed by a sequence of $\Omega_{2}^{\downarrow}$ moves. 
In the proof process, the author proved that, link diagrams that are transformed into each other by applying finitely many Reidemeister moves II and III are always transformed into each other by applying a sequence of $\Omega_{2}^{\uparrow}$ moves, followed by a sequence of $\Omega_{3}$ moves, followed by a sequence of $\Omega_{2}^{\downarrow}$ moves. Next question is natural from this point of view. Is there such a simple sequence between link diagrams that are transformed into each other by applying finitely many Reidemeister moves I and III? In this paper, we discuss such simple sequences. There are some related papers in this topic (see \cite{upper_bound}, \cite{sequence}, \cite{lower}, \cite{bound}, \cite{minimal}).

In Section \ref{sec:preliminaries}, we prepare definitions and theorems that are necessary for the main theorems in this paper. In Section \ref{sec:ordinary_ordered}, we prove that, infinitely many pairs of trivial knot diagrams that are transformed into each other by applying Reidemeister moves I and III are NOT transformed into each other by a sequence of Reidemeister moves I, followed by a sequence of Reidemeister moves III, followed by a sequence of Reidemeister moves I. To create a simple sequence between link diagrams that are transformed into each other by applying finitely many Reidemeister moves I and III, in Section \ref{sec:I-generalized}, we prove that the link diagrams are always transformed into each other by applying a sequence of $\Omega_{1}^{\uparrow}$ moves, followed by a sequence of $\Omega_{3}^{*}$ moves (which are $\Omega_{3}$ moves modified by $\Omega_{1}$ moves), followed by a sequence of $\Omega_{1}^{\downarrow}$ moves.

\section{Premilinaries}\label{sec:preliminaries}
In this section, we check definitions and theorems that are necessary for the main theorems in this paper.

The following theorem is well known in knot theory.

\begin{theorem}[K. Reidemeister \cite{reidemeister}, 1927] \label{theo:reidemeister}
Every pair of diagrams of equivalent links can be transformed into each other by applying finitely many of the local moves appeared in Figure \ref{fig:reidemeister}.
Furthermore, two link diagrams that are transformed into each other by applying finitely many of the local moves appeared in Figure \ref{fig:reidemeister} represent equivalent links.
\end{theorem}

\begin{figure}[H]
 \begin{center}
 \includegraphics[width=9cm,keepaspectratio]{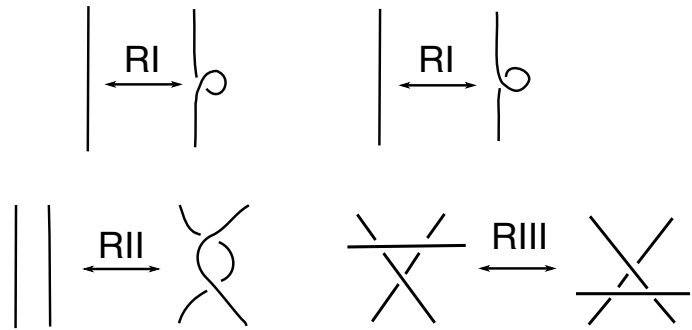}
 \end{center}
 \caption{Reidemeister moves}
 \label{fig:reidemeister}
\end{figure}

In 2006, A. Coward proved the following theorem.

\begin{theorem}[A. Coward \cite{ordered_sequence}, 2006] \label{theo:orderedtheorem}
Given two diagrams $D_{1}$ and $D_{2}$ for a link $L$, $D_{1}$ may be turned into $D_{2}$ by a sequence of $\Omega^{\uparrow}_{1}$ moves, followed by a sequence of $\Omega^{\uparrow}_{2}$ moves, followed by a sequence of ${\Omega}_{3}$ moves, followed by a sequence of $\Omega^{\downarrow}_{2}$ moves.\\[10pt]
Furthermore, if $D_{1}$ and $D_{2}$ are diagrams of a link where the winding number and framing of each component is the same in each diagram, then $D_{1}$ may be turned into $D_{2}$ by a sequence of $\Omega^{\uparrow}_{2}$ moves, followed by a sequence of $\Omega_{3}$ moves, followed by a sequence of $\Omega^{\downarrow}_{2}$ moves.
\end{theorem}

\section{Ordinary ordered sequences between RI-III related link diagrams}\label{sec:ordinary_ordered}
In this section, we prove that infinitely many pairs of trivial knot diagrams that are transformed into each other by applying Reidemeister moves I and III are NOT transformed into each other by a sequence of the Reidemeister moves I that increase the number of crossings, followed by a sequence of Reidemeister moves III, followed by a sequence of the Reidemeister moves I that decrease the number of crossings. 

\begin{definition}
\label{def:RI-III_related}
\emph{RI-III related} link diagrams are defined to be link diagrams that are transformed into each other by applying finitely many Reidemeister moves I and III.  
\end{definition}

\begin{definition}
\emph{A region} is defined to be a connected complement of link diagrams in $\mathbb{S}^{2}$. \emph{Adjacent crossings to a region} are defined to be crossings that are adjacent to the region. For instance, in Figure \ref{fig:trivial_knot_diagram}, the region with a yellow color has 4 adjacent crossings $c_{1}$, $c_{2}$, $c_{3}$, $c_{4}$.

\begin{figure}[H]
 \begin{center}
   \includegraphics[width=0.4\linewidth,keepaspectratio]{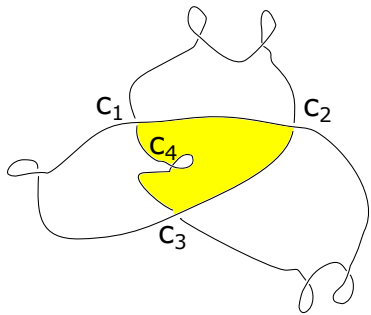}
 \end{center}
 \caption{a trivial knot diagram}
 \label{fig:trivial_knot_diagram}
\end{figure}
\end{definition}

\begin{definition}
\emph{A colored link diagram} is defined to be a link diagram that is decorated with gray disks and black squares as in Figure \ref{fig:counter_example}. \emph{Colored Reidemeister moves I and III} are defined to be Reidemeister moves I and III with gray disks and black squares as in Figure \ref{fig:colored_Reidemeister}. Note that, in Figure \ref{fig:colored_Reidemeister}, the outsides of the local disks are the same including the locations of gray disks and black squares, and that the local disks of $\Omega_{1}$ moves located left below in Figure \ref{fig:colored_Reidemeister} are covered by gray disks. We use colored link diagrams and colored Reidemeister moves defined here to prove Theorem \ref{theo:counter_theorem}.

\begin{figure}[H]  
  \begin{center}
\includegraphics[width=11cm,keepaspectratio]{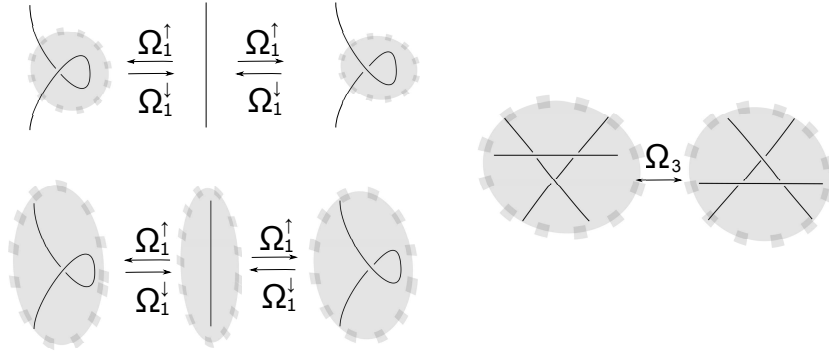}
  \end{center}
  \caption{colored $\Omega_{1}$ moves and colored $\Omega_{3}$ move}
 \label{fig:colored_Reidemeister}
 \end{figure}  
\end{definition}

\begin{theorem}\label{theo:counter_theorem}
There exists infinitely many pairs of RI-III related trivial knot diagrams that may NOT be transformed into each other by applying a sequence of $\Omega^{\uparrow}_{1}$ moves, followed by a sequence of $\Omega_{3}$ moves, followed by a sequence of $\Omega^{\downarrow}_{1}$ moves.
\end{theorem}

\begin{proof}

We first consider the trivial knot diagram as in Figure \ref{fig:counter_example}. 

\begin{figure}[H]
 \begin{center}
   \includegraphics[width=0.4\linewidth,keepaspectratio]{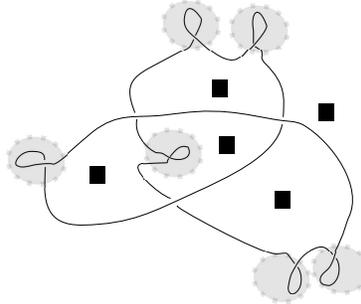}
 \end{center}
 \caption{trivial knot diagram that produces the examples of Theorem \ref{theo:counter_theorem}}
 \label{fig:counter_example}
\end{figure}

The diagram in Figure \ref{fig:counter_example} is decorated with gray disks and black squares. Each gray disk in the trivial knot diagram contains a monogon (i.e. the shape of Reidemeister move I). Each black square indicates one region of the diagram. By deleting all the crossings in the gray disks with $\Omega^{\downarrow}_{1}$ moves, we may easily see that there is a sequence of Reidemeister moves I and III that transforms the diagram depicted in Figure \ref{fig:counter_example} into the trivial knot diagram without a crossing. 

As in Figure \ref{fig:counter_example2}, we consider $n$-times connected sums of the trivial knot diagram depicted in Figure \ref{fig:counter_example}. Note that every region with a black square in the diagram has greater than or equal to $4$ adjacent crossings.

\begin{figure}[H]
 \begin{center}
   \includegraphics[width=11cm,keepaspectratio]{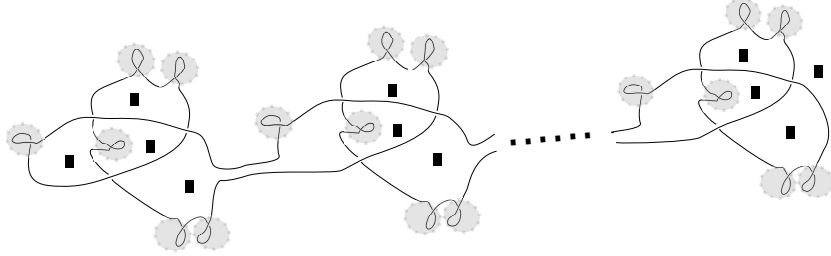}
 \end{center}
  \caption{$n$-times connected sums of the trivial knot diagram in Figure \ref{fig:counter_example}}
  \label{fig:counter_example2}
  \end{figure}

 We assume that, the trivial knot diagram depicted in Figure \ref{fig:counter_example2} that has neither gray disks nor black squares may be turned into the trivial knot diagram without a crossing by a sequence of $\Omega_{1}^{\uparrow}$ moves, followed by a sequence of $\Omega_{3}$ moves, followed by a sequence of $\Omega_{1}^{\downarrow}$ moves. To clarify the contradiction of the assumption, we are going to add gray disks and black squares to the trivial knot diagram we assumed above as in Figure \ref{fig:counter_example2}, and replace the ordinary ordered sequence with the ordinary ordered sequence of colored Reidemeister moves, which is the same sequence as the previous ordinary ordered sequence when we ignore gray disks and black squares. We may replace $\Omega_{1}^{\uparrow}$ moves in the ordinary ordered sequence with colored $\Omega_{1}$ moves which are the same moves as the $\Omega_{1}^{\uparrow}$ moves when we ignore gray disks and black squares, since colored $\Omega_{1}^{\uparrow}$ moves defined in Figure \ref{fig:colored_Reidemeister} have all possibilities of locations of gray disks. (To be clear, when the local 
 disk where we apply an $\Omega_{1}^{\uparrow}$ move in the ordinary ordered sequence is covered by a gray disk, we apply the colored $\Omega_{1}^{\uparrow}$ move located at the lower left in Figure \ref{fig:colored_Reidemeister} that is the same move as the $\Omega_{1}^{\uparrow}$ move when we ignore gray disks and black squares. When the local disk where we apply an $\Omega_{1}^{\uparrow}$ move in the ordered sequence is not covered by a gray disk, by adjusting the location of the local disk, the local disk may be taken as not being covered by a gray disk. In this case, we apply the colored $\Omega_{1}^{\uparrow}$ move located at the upper left in Figure \ref{fig:colored_Reidemeister} that is the same move as the $\Omega_{1}^{\uparrow}$ move when we ignore gray disks and black squares.) By applying the sequence of colored $\Omega_{1}^{\uparrow}$ moves as written above, the resulting colored diagram is as in Figure \ref{fig:counter_example4}, where finitely many gray disks are added and any region with a black square is NOT covered by a gray disk.

\begin{figure}[H]
 \begin{center}
   \includegraphics[width=12cm,keepaspectratio]{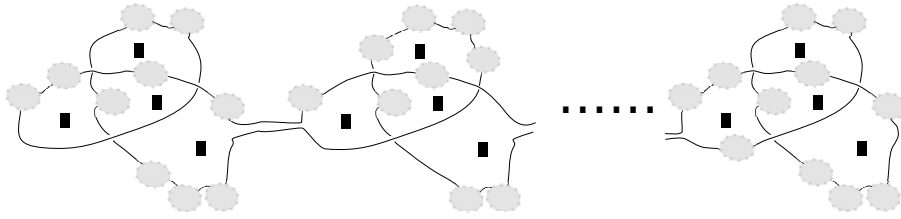}
 \end{center}
  \caption{the resulting colored trivial knot diagram after applying the sequence of colored $\Omega_{1}^{\uparrow}$ moves}
  \label{fig:counter_example4}
  \end{figure}

$\Omega_{3}$ moves in the ordinary ordered sequence may be replaced with colored $\Omega_{3}$ moves which are the same moves as the $\Omega_{3}$ moves when we ignore gray disks and black squares, since every region with a black square have greater than or equal to $4$ vertices in the process and the $\Omega_{3}$ moves occur only inside the gray disks.  

$\Omega_{1}^{\downarrow}$ moves in the ordinary ordered sequence may be replaced with colored $\Omega_{1}^{\downarrow}$ moves which are the same moves as the $\Omega_{1}^{\downarrow}$ moves when we ignore gray disks and black squares, since every region with a black square has greater than or equal to $2$ vertices in the process and the $\Omega_{1}^{\downarrow}$ moves occur only inside the gray disks.

In the whole process written above, every region with a black square does not vanish.
This fact contradicts the assumption that the trivial knot diagram depicted in Figure \ref{fig:counter_example2} (that has neither gray disks nor black squares) may be turned into the trivial knot diagram without a crossing by a sequence of $\Omega_{1}^{\uparrow}$ moves, followed by a sequence of $\Omega_{3}$ moves, followed by a sequence of $\Omega_{1}^{\downarrow}$ moves. Thus, there is no ordinary ordered sequence that transforms the diagram depicted in Figure \ref{fig:counter_example2} into the trivial knot diagram without a crossing.

\end{proof}

\section{I-generalized ordered sequences between RI-III related link diagrams}\label{sec:I-generalized}
In this section, we define an I-generalized ordered sequence and prove that RI-III related link diagrams are always transformed into each other by applying an I-generalized ordered sequence.
 
\begin{definition}\label{def:rectangle}

\emph{A rectangle}, as in Figure \ref{fig:rectangle}, is defined to be an $1$-tangle diagram that is transformed from a trivial $1$-tangle diagram without a crossing by applying finitely many $\Omega_{1}^{\uparrow}$ moves in this paper. \emph{A Reidemeister move III*} is defined to be a Reidemeister move III with three rectangles as in Figure \ref{fig:rectangles}. Note that, in Figure \ref{fig:rectangles}, the two rectangles with the same color are equivalent as $1$-tangle diagrams. See Figure \ref{fig:RIII_example} for an example of a Reidemeister move III*.

\begin{figure}[H]
\begin{center}
\includegraphics[width=0.7\linewidth,clip]{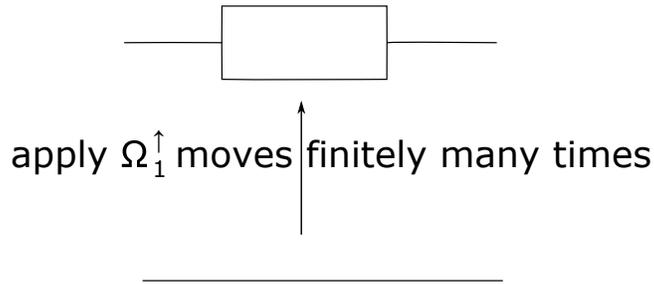}
\end{center}
\caption{A \emph{rectangle} drawn in this figure is defined to be an $1$-tangle diagram that is transformed from a trivial $1$-tangle diagram without a crossing by applying finitely many $\Omega_{1}^{\uparrow}$ moves }
\label{fig:rectangle}
\end{figure}

\begin{figure}[H]
\begin{center}
\includegraphics[width=9cm,keepaspectratio]{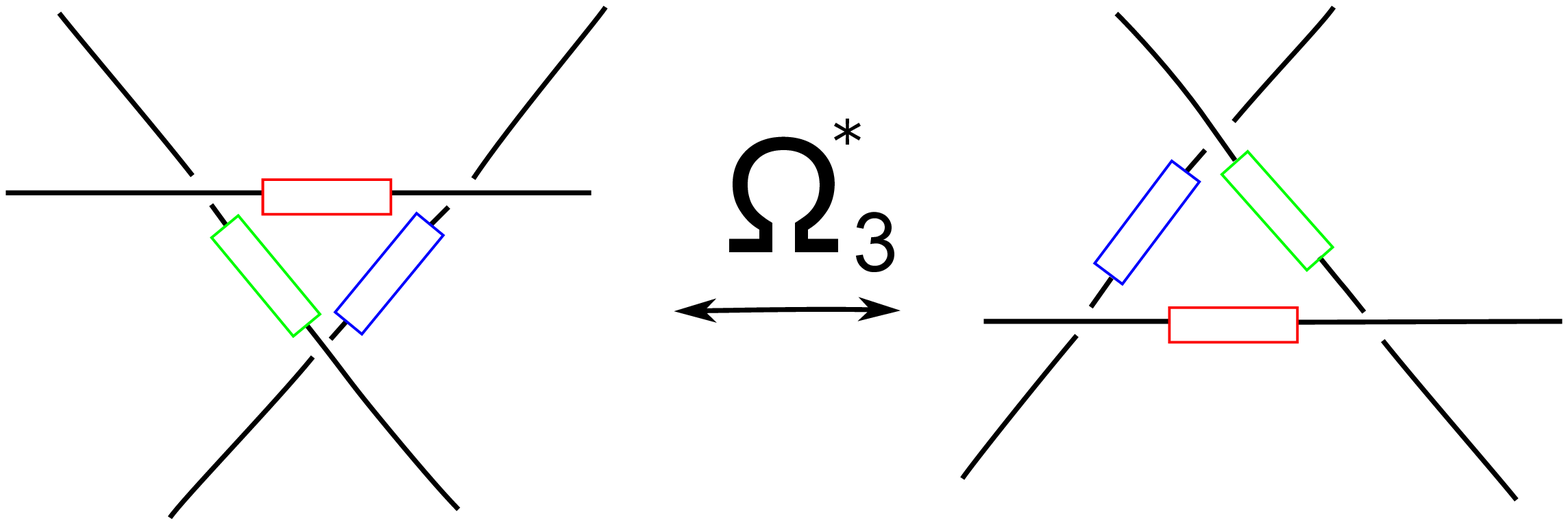}
\end{center}
\caption{An $\Omega_{3}$ move with three rectangles}
\label{fig:rectangles}
\end{figure}

\begin{figure}[H]
\begin{center}
\includegraphics[width=9cm,keepaspectratio]{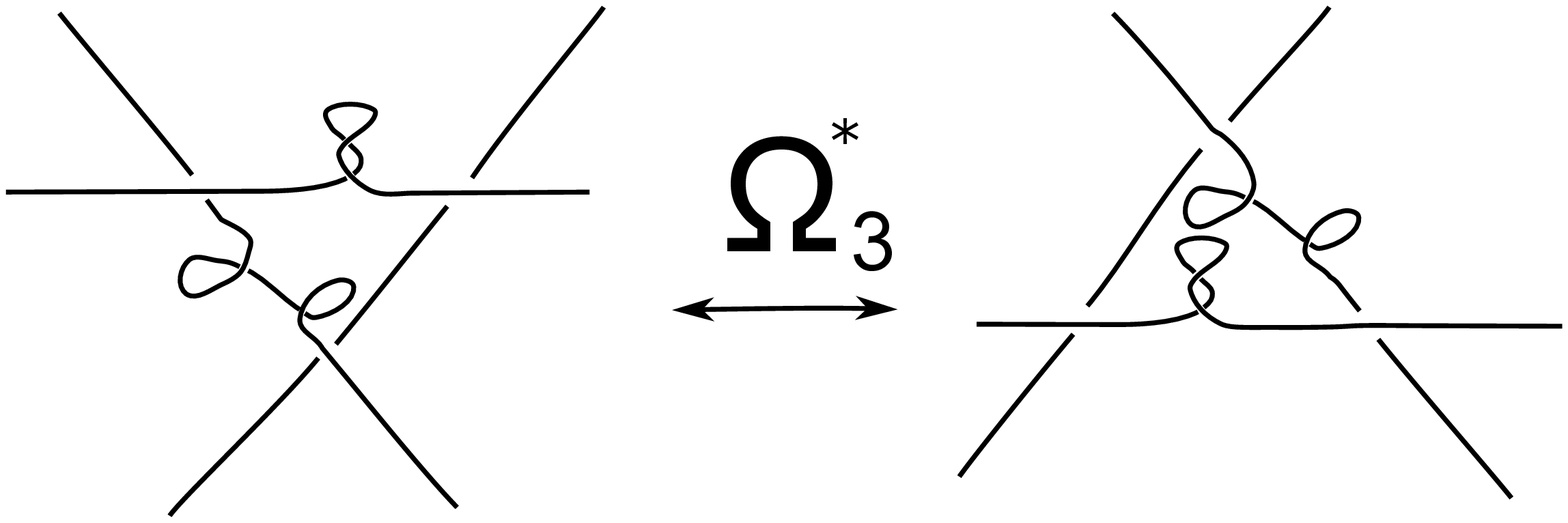}
\end{center}
\caption{An example of an $\Omega_{3}^{*}$ move}
\label{fig:RIII_example}
\end{figure}

\end{definition}

\begin{definition}\label{def:I-generalized}
\emph{An I-generalized ordered sequence} is defined to be a sequence of $\Omega_{1}^{\uparrow}$ moves, followed by a sequence of $\Omega_{3}^{*}$ moves, followed by a sequence of $\Omega_{1}^{\downarrow}$ moves in this paper. Note that the definition of an I-generalized ordered sequence includes no move between the same two link diagrams. 
\end{definition}

\begin{lemma}\label{lem:move_1}
Let $D$ be a link diagram. Let $D_{1}$ and $D_{2}$ be two link diagrams, each of which is transformed from $D$ by an $\Omega_{1}^{\uparrow}$ move. Then, there exists a link diagram $D_{3}$ that is transformed from $D_{1}$ and $D_{2}$ by an $\Omega_{1}^{\uparrow}$ move. See Figure $\ref{fig:1-moves}$ (a) for the overview.

\end{lemma}

\begin{proof}
An $\Omega_{1}^{\uparrow}$ move may be applied at a trivial $1$-tangle diagram without a crossing. This means that there are two trivial $1$-tangle diagrams without a crossing in $D$, where the two $\Omega_{1}^{\uparrow}$ moves are applied. The two $\Omega_{1}^{\uparrow}$ moves may be applied at the same time, since the two local disks may be taken as disjoint. Thus, Lemma \ref{lem:move_1} has been proven.

\end{proof}

\begin{lemma}\label{lem:move_1_seq}
Let $D$ be a link diagram. Let $D_{1}$, $D_{2}$ be two link diagrams that are transformed from $D$ by an $\Omega_{1}^{\uparrow}$ move and $n$ $\Omega_{1}^{\uparrow}$ moves (n$\in$$\mathbb{N}$), respectively. Then, there exists a link diagram $D_{3}$ that is transformed from $D_{1}$ by $n$ $\Omega_{1}^{\uparrow}$ moves and from $D_{2}$ by an $\Omega_{1}^{\uparrow}$ move. See Figure $\ref{fig:1-moves}$ (b) for the overview.

\end{lemma}

\begin{proof}
By using Lemma \ref{lem:move_1} $n$ times, we can prove Lemma \ref{lem:move_1_seq} easily.  
\end{proof}

\begin{figure}[H]
\centering
\begin{subfigure}{0.49\textwidth}
\includegraphics[trim=-4cm 0 0 0, height=3.5cm, keepaspectratio]{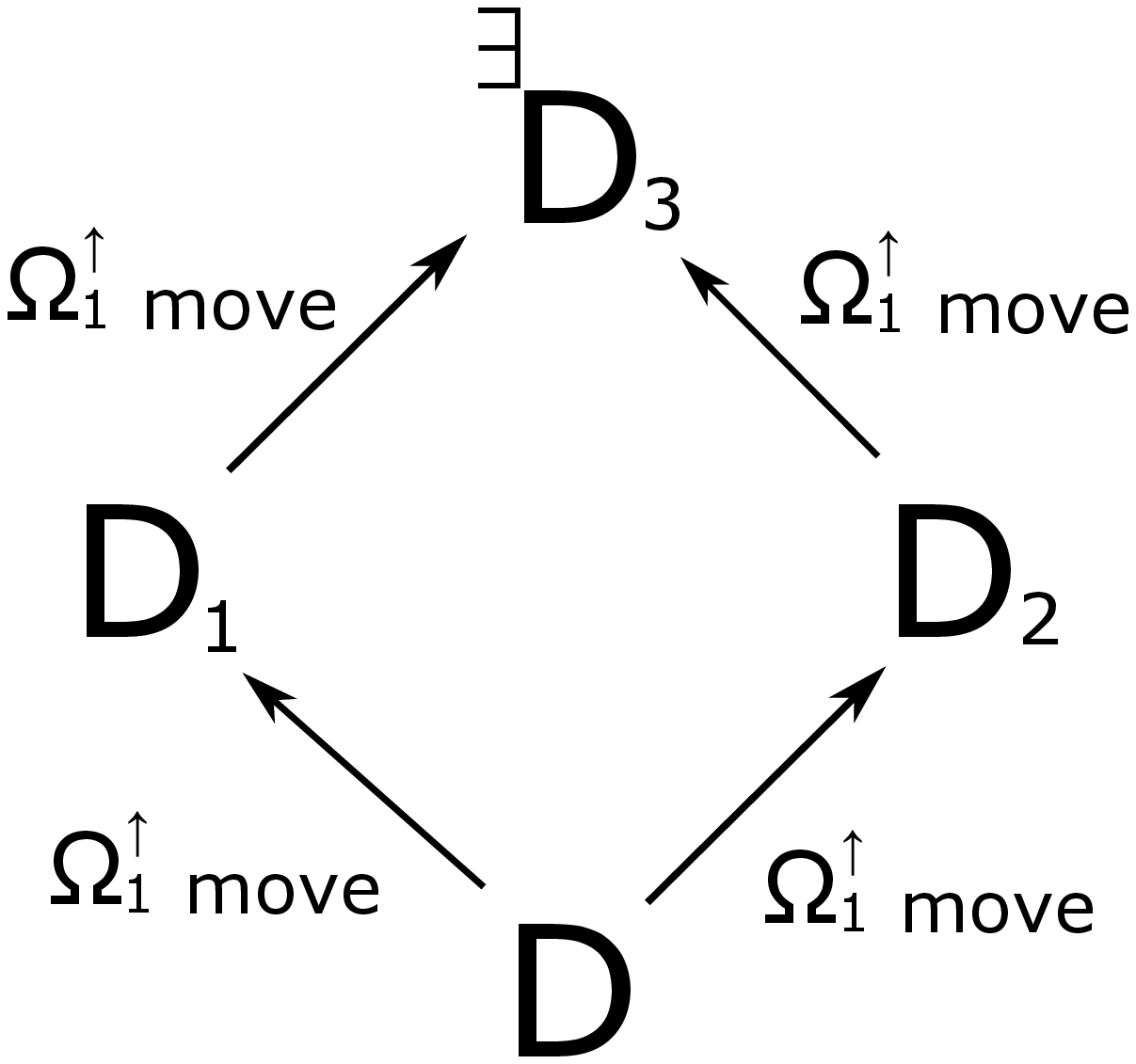}
\caption{the overview of Lemma \ref{lem:move_1}}
\end{subfigure}
\begin{subfigure}{0.49\textwidth}
\includegraphics[trim=-4cm 0 0 0, height=3.5cm, keepaspectratio]{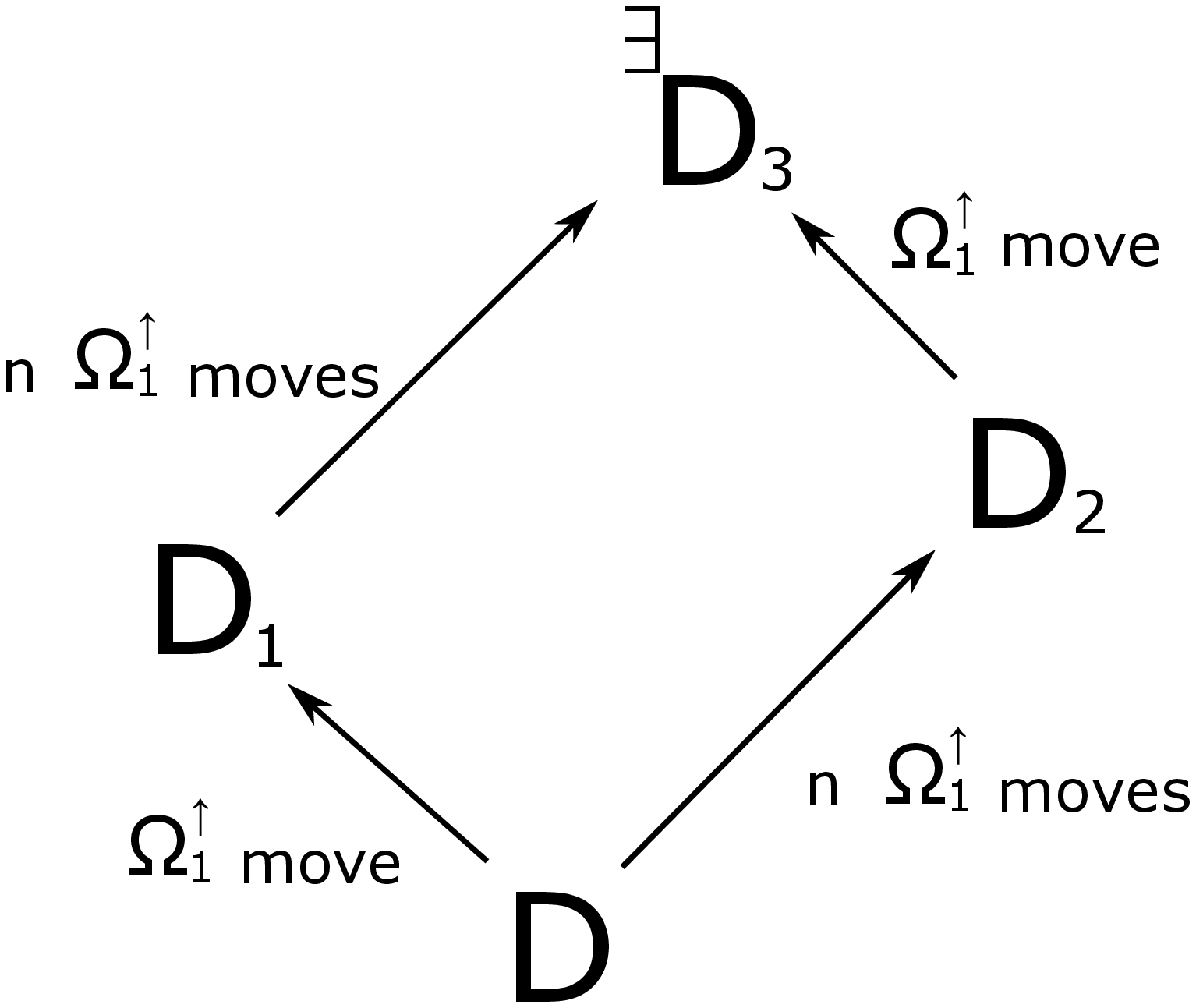}
\caption{the overview of Lemma \ref{lem:move_1_seq}}
\end{subfigure}
\caption{}
\label{fig:1-moves}
\end{figure}

\begin{lemma}\label{lem:move_3}
Let $D$ be a link diagram. Let $D_{1}$, $D_{2}$ be two link diagrams that are transformed from $D$ by an $\Omega_{1}^{\uparrow}$ move and an $\Omega_{3}$ move, respectively. There exists a link diagram $D_{3}$ that is transformed from $D_{1}$ by an $\Omega_{3}^{*}$ move and from $D_{2}$ by an $\Omega_{1}^{\uparrow}$ move. See Figure \ref{fig:3-moves} (a)  for the overview.

\end{lemma}

\begin{proof}
When the two local regions of the $\Omega_{1}^{\uparrow}$ move and the $\Omega_{3}$ move in $D$ may be taken as disjoint, we may apply the two local moves separately. This means that the resulting diagram may be taken as $D_{3}$.

When the two local regions of the $\Omega_{1}^{\uparrow}$ move and the $\Omega_{3}$ move in $D_{1}$ may not be taken as disjoint, the local region of the $\Omega_{1}^{\uparrow}$ move in $D_{1}$ may be taken as a rectangle. We may consider the rectangle and the local disk of the $\Omega_{3}$ move in $D_{1}$ as a new $\Omega_{3}^{*}$ move from $D_{1}$ (to $D_
{3}$). The difference between the $D_{3}$ appeared right above and $D_{2}$ is just an $\Omega_{1}^{\uparrow}$ move. Thus, Lemma \ref{lem:move_3} has been proven.
\end{proof}

\begin{remark}\label{rem:move_3}
Even if we change the condition of $D_{2}$ in Lemma \ref{lem:move_3} into that $D_{2}$ is a link diagram that is transformed from $D$ with 'an $\Omega_{3}^{*}$ move', we may prove the Lemma in the same way as Lemma \ref{lem:move_3}. 
\end{remark}

\begin{lemma}\label{lem:move_3_seq}
Let $D$ be a link diagram. Let $D_{1}$, $D_{2}$ be two link diagrams that are transformed from $D$ by $n$ $\Omega_{1}^{\uparrow}$ moves and $m$ $\Omega_{3}$ moves, respectively. Then, there exists a link diagram $D_{3}$ that is transformed from $D_{1}$ by $m$ $\Omega_{3}^{*}$ moves and from $D_{2}$ by $n$ $\Omega_{1}^{\uparrow}$ moves. See Figure \ref{fig:3-moves} (b) for the overview.
\end{lemma}

\begin{proof}
By applying Lemma \ref{lem:move_3} and Remark \ref{rem:move_3} $n$$\times$$m$ times, we may prove Lemma \ref{lem:move_3_seq} easily.

\end{proof}

\begin{figure}[H]
\centering
\begin{subfigure}{0.49\textwidth}
\includegraphics[height=3cm, keepaspectratio]{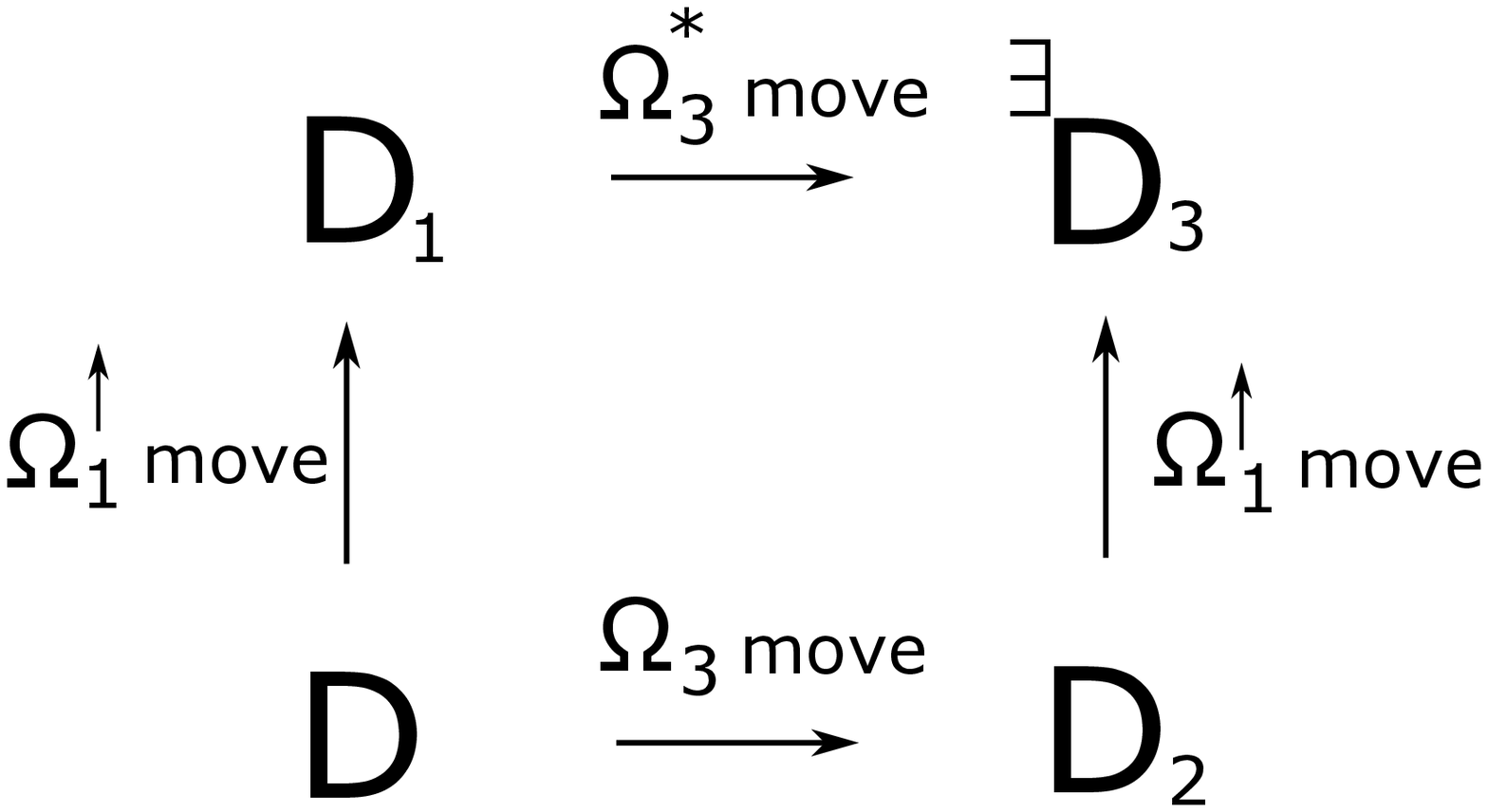}
\caption{the overview of Lemma \ref{lem:move_3}}
\end{subfigure}
\begin{subfigure}{0.49\textwidth}
\centering
\includegraphics[height=3cm, keepaspectratio]{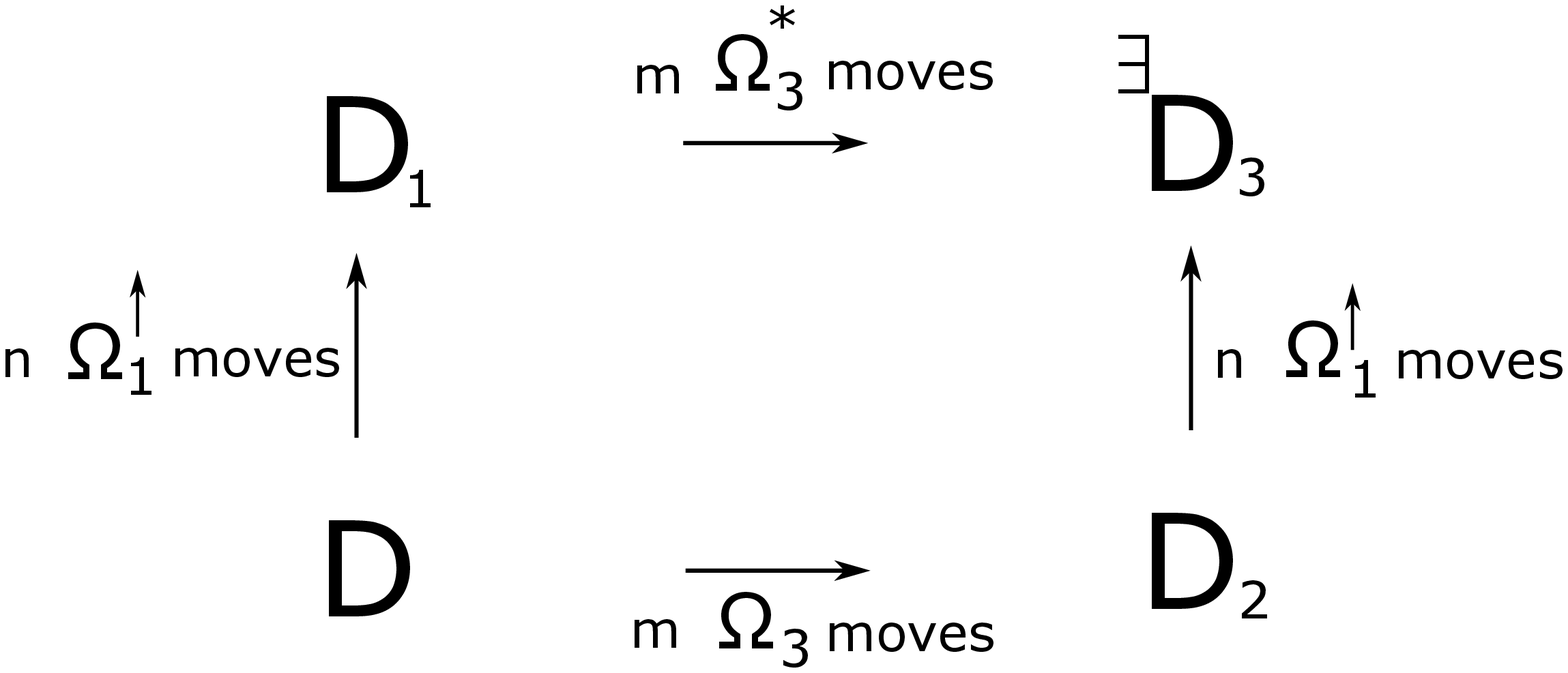}
\caption{the overview of Lemma \ref{lem:move_3_seq}}
\end{subfigure}
\caption{}
\label{fig:3-moves}
\end{figure}

\begin{remark}
Even if we change the condition of $D_{2}$ in Lemma \ref{lem:move_3_seq} into that $D_{2}$ is a link diagram that is transformed from $D$ by '$n$ $\Omega_{3}^{*}$ moves', we may prove the Lemma in the same way as Lemma \ref{lem:move_3_seq}. 

\end{remark}

\begin{theorem}\label{theo:generalized_ordered}
If two link diagrams $D_{1}$ and $D_{2}$ are transformed into each other by applying finitely many Reidemeister moves I, III, then there exists an I-generalized ordered sequence that transforms $D_{1}$ into $D_{2}$. Moreover, the I-generalized ordering sequence has the same numbers of $\Omega_{1}^{\uparrow}$ moves and $\Omega_{3}^{*}$ moves and $\Omega_{1}^{\downarrow}$ moves as the numbers of $\Omega_{1}^{\uparrow}$ moves and $\Omega_{3}$ moves and $\Omega_{1}^{\downarrow}$ moves in the sequence of finitely many Reidemeister moves I, III that we assumed first, respectively.

\end{theorem}

\begin{proof}
By Definition \ref{def:I-generalized}, the same two link diagrams are an example of two link diagrams that have an I-generalized ordered sequence (i.e. no move) between them. 

Let $D_{1}$, $D_{2}$ be two link diagrams that are transformed into each other by applying an I-generalized ordered sequence. Let $\mathcal{D}_{3}$ be the set of link diagrams that are transformed from $D_{2}$ by an $\Omega_{1}^{\uparrow}$ move or an $\Omega_{1}^{\downarrow}$ move or an $\Omega_{3}$ move. From here onward, to prove Theorem \ref{theo:generalized_ordered}, we are going to check that there exists a new I-generalized ordered sequence that transforms $D_{1}$ into every diagram of $\mathcal{D}_{3}$.

When a diagram $D_{3}$ of $\mathcal{D}_{3}$ is transformed from $D_{2}$ with an $\Omega_{1}^{\downarrow}$ move, the sequence that consists of the I-generalized ordered sequence we assumed right above and the $\Omega_{1}^{\downarrow}$ move is the new I-generalized ordered sequence that transforms $D_{1}$ into $D_{3}$. Moreover, in this argument, since $\Omega_{1}^{\downarrow}$ move is added to both the sequence of finitely many Reidemeister moves I, III and the I-generalized ordered sequence, the numbers of $\Omega_{1}^{\downarrow}$ moves in the two resulting sequences remain the same. In this argument, the numbers of $\Omega_{1}^{\downarrow}$ moves and $\Omega_{3}^{*}$ moves and $\Omega_{3}$ moves do not change.

When a diagram $D_{3}$ of $\mathcal{D}_{3}$ is transformed from $D_{2}$ with an $\Omega_{1}^{\uparrow}$ move, by using Lemma \ref{lem:move_1_seq} and Lemma \ref{lem:move_3_seq} repeatedly on the I-generalized
ordered sequence that transforms $D_{1}$ into $D_{2}$ and the $\Omega_{1}^{\uparrow}$ move that transforms $D_{2}$ into $D_{3}$, we may construct a new I-generalized sequence that transforms $D_{1}$ into $D_{3}$.  Moreover, in this argument, since an $\Omega_{1}^{\uparrow}$ move is added to both the sequence of finitely many Reidemeister moves I, III and the I-generalized ordered sequence, the numbers of $\Omega_{1}^{\uparrow}$ moves in the two resulting sequences remain the same. In this argument, the numbers of $\Omega_{1}^{\downarrow}$ moves and $\Omega_{3}^{*}$ moves and $\Omega_{3}$ moves do not change.

When a diagram $D_{3}$ of $\mathcal{D}_{3}$ is transformed from $D_{2}$ by an $\Omega_{3}$ move, by using Lemma \ref{lem:move_3} repeatedly, we may construct a new I-generalized ordered sequence that transforms $D_{1}$ into $D_{3}$. Moreover, in this argument, since the $\Omega_{3}$ move is added to the sequence of finitely many Reidemeister moves I, III and an $\Omega_{3}^{*}$ move is added to the I-generalized ordered sequence, the number of $\Omega_{3}$ moves and the number of $\Omega_{3}^{*}$ moves remain the same. In this argument, the numbers of $\Omega_{1}^{\downarrow}$ moves and $\Omega_{3}^{*}$ moves and $\Omega_{3}$ moves do not change.

Thus, The proof of Thoerem \ref{theo:generalized_ordered} is completed.
\end{proof}

\end{document}